\documentclass[ reqno]{amsart}
\usepackage{amsmath, amsthm, amscd, amsfonts, amssymb, graphicx, color}
\usepackage[bookmarksnumbered, colorlinks, plainpages]{hyperref}
\usepackage{enumerate}
\usepackage{cite}

\newtheorem{theorem}{Theorem}[section]
\newtheorem{lemma}[theorem]{Lemma}

\newtheorem{corollary}[theorem]{Corollary}
\theoremstyle{definition}

\newtheorem{example}[theorem]{Example}

\theoremstyle{remark}
\newtheorem{remark}[theorem]{Remark}
\numberwithin{equation}{section}
\makeatletter
\everymath\expandafter{\the\everymath \displaystyle}
\makeatother

\begin{document}

\allowdisplaybreaks 
\setcounter{page}{1}

\title[Sharp operator mean inequalities of the numerical radii]{Sharp  operator mean inequalities of the numerical radii}
\author [H. Jafarmanesh and M. Khosravi] {Hosna Jafarmanesh and Maryam Khosravi} 
\address{Hosna Jafarmanesh: Department of Mathematics and Computer Sciences,
Hakim Sabzevari University, Sabzevar, P.O. Box 397, Iran.}
\email{\textrm{Hosna.jafarmanesh@yahoo.com}}
\address{Maryam Khosravi: Department of Pure Mathematics, Faculty of Mathematics and Computer, Shahid Bahonar  University of Kerman, Kerman, Iran.}
\email{\textrm{khosravi@uk.ac.ir}}
.
\begin{abstract}
We present several sharp upper bounds and some extension for product operators. Among other inequalities, it is shown that if $ 0<mI \leq B^{*}f^{2}(|X|)B$, $A^{*}g^{2}(|X^{*}|)A \leq MI$, $f, g$ are non-negative continuous functions on $[0, \infty)$ such that $f(t)g(t)=t$, $(t\geq 0)$, then for all non-negative operator monotone decreasing function $h$ on $[0, \infty)$, we obtain that
\begin{align*}
\left\|h\big(B^{*}f^{2}(|X|)B\big) \sigma h\big(A^{*}g^{2}(|X^{*}|)A\big)\right\| \leq \frac{m k}{M}  h\left(\left| \left\langle (A^{*}XB)x,x\right\rangle \right|\right),
\end{align*}
As an application of the above inequality, it is shown that
\begin{align*}
\omega\big(A^{*}XB\big) \leq  \frac{m k}{M} \left\|B^{*}f^{2}(|X|)B  !  A^{*}g^{2}(|X^{*}|)A\right\|,
\end{align*}
where, $k=\dfrac{(M+m)^{2}}{4mM}$ and $! \leq \sigma \leq \bigtriangledown$. 
\end{abstract}

\maketitle

\section{Introduction}
Let $\mathcal{B}(\mathcal{H})$ denote the $C^{*}$-algebra of all bounded linear operators on a complex Hilbert space $\mathcal{H}$. An operator $A\in \mathcal{B}(\mathcal{H})$ is called positive if $\langle Ax,x\rangle\geq 0$ for all $x\in \mathcal{H}$. We write $A\geq 0$ if $A$ is positive.\\ 
A continuous real-valued function $f$ defined on interval $J$ is said to be operator monotone increasing (decreasing) if for every two positive operators $A$ and $B$ with spectral in $J$, the inequality $A \leq B$ implies $f(A) \leq f(B)$ ($f(A) \geq f(B)$), respectively. As an example, it is well known that the power function $x^{p}$ on $(0,\infty)$ is operator monotone increasing if $p\in[0,1]$ and operator monotone decreasing if and only if $p\in[-1, 0]$.\\ 
If $f:J \to \mathbb{R}$ is a convex function and $A$ is a self-adjoint operator with spectrum in $J$, then
\begin{align} \label{eq 002}
           f\big(\langle Ax,x\rangle\big) \leq \langle f(A)x,x\rangle.
\end{align} 
for each $x\in\mathcal{H}$ with $\|x\|=1$, and the reverse inequality holds if $f$ is concave (see \cite{Mond}).\\
 The spectral radius and the numerical radius of $A\in \mathcal{B}(\mathcal{H})$ are defined by $r(A)=\sup\{ |\lambda| ~:~\lambda\in sp(A)\}$ and
\[\omega(A)=\sup\{ |\langle Ax,x\rangle |~:~x\in H,~\|x\|=1\},\]
respectively. It is well-known that $r(A) \leq \omega(A)$ and $\omega(.)$ defines a norm on $\mathcal{B}(\mathcal{H})$, which is equivalent to the usual operator norm $\|.\|$.\\ In fact, for any $ A\in \mathcal{B}(\mathcal{H})$,  
\begin{equation}\label{eq1.1}
\dfrac{1}{2}\|A\|\leq \omega(A)\leq\|A\|.
\end{equation}  
Kittaneh \cite{Kittaneh 2} has shown that for $A\in\mathcal{B}(\mathcal{H})$,
\begin{align} \label{eq kit2}
\omega^{2}(A) \leq \frac{1}{2} \||A|^{2} + |A^{*}|^{2}\|,
\end{align}
which is a refinement of right hand side of inequality \eqref{eq1.1}.\\
 Dragomir \cite{Dragomir} proved that for any $A, B \in \mathcal{B}(\mathcal{H})$ and for all $p\geq 1$,
\begin{align} \label{eq dra04}
\omega^{p}(B^{*}A) \leq \frac{1}{2} \|(A^{*}A)^{p} + (B^{*}B)^{p}\|.
\end{align}
In \cite{Sattari}, it has been shown that if $A, B \in \mathcal{B}(\mathcal{H})$ and $p\geq 1$, then
\begin{align} \label{Sattari}
\omega^{p} (B^{*} A) \leq \frac{1}{4} \| (A A^{*})^{p} +(B B^{*})^{p}\| +\frac{1}{2} \omega^{p}(A B^{*}),
\end{align}
which is generalization of inequality \eqref{eq dra04} and in  particular cases is sharper than this inequality.
Shebrawi et al. \cite{Shebrawi} generalized inequalities \eqref{eq kit2} and \eqref{eq dra04}, as follows:\\
If $A, B, X \in \mathcal{B}(\mathcal{H})$ and $p\geq1$, we have 
\begin{align} \label{Shebrawi}
\omega^{p} (A^{*} X B) \leq \frac{1}{2} \| (A^{*} |X^{*}| A)^{p} + (B^{*} |X| B)^{p} \|.
\end{align}

In this paper, we first derive a new lower bound for inner-product of products $A^{*}XB$ involving operator monotone decreasing function, and, so we give refinement of the inequalities \eqref{eq dra04} and \eqref{Shebrawi}. We prove a numerical radius, which is similar to \eqref{Sattari} in some example is sharper than \eqref{Sattari}.\\
 In particular, we extend inequality \eqref{Sattari} and also find some example which show that is a refinement of \eqref{Shebrawi}. In the next, we present numerical radius inequalities for products of operators, which one of the applications of our results is a generalization of \eqref{eq kit2}.

\section{Main results}\label{se2}

We first recall that for positive invertible operators $A, B \in \mathcal{B}(\mathcal{H})$, the weighted operator arithmetic and harmonic means are defined, by \\
\[A \bigtriangledown_{\nu} B = (1-\nu) A + \nu B \] and \[A !_{\nu} B = \big((1-\nu) A^{-1} + \nu B^{-1})\big)^{-1}.\] \\
It is well-known that if $\sigma_{\nu}$ is a symetric operator mean, then
\begin{align*} 
A !_{\nu} B \leq A \sigma_{\nu} B \leq A \bigtriangledown_{\nu} B. 
 \end{align*}

To prove our numerical radius inequalities, we need several  lemmas.

\begin{lemma} \cite{Kittaneh 0} \label{le k2}
 If $A\in B(\mathcal{H})$ and $f, g$ are non-negative continuous functions on $[0, \infty)$ satisfying $f(t)g(t)=t$, $(t\geq 0)$, then for each $x, y\in \mathcal{H}$ 
\begin{align*}
    |\langle Ax,y\rangle| \leq \|f(|A|)x\|  \|g(|A^{*}|)y\|.   
\end{align*}
\end{lemma}
\begin{lemma}\cite{Jafarmanesh} \label{eq j}
Let $0<mI \leq A, B \leq MI$, $0\leq \nu\leq1$, $!_\nu \leq \tau_\nu, \sigma_\nu \leq \bigtriangledown_\nu$ and $\Phi$ be a positive unital linear map. If $h$ is an operator monotone decreasing function on $(0, \infty)$, then
\begin{align*}
h\big(\Phi(A)\big) \sigma_\nu h\big(\Phi(B)\big) \leq k h\big(\Phi(A \tau_\nu B)\big) 
\end{align*}
where, $k=\dfrac{(M+m)^{2}}{4mM}$ stands for the known Kantorovich constant.
\end{lemma}
\begin{lemma} \label{le h2}
Let $A\in\mathcal{B}(\mathcal{H})$ be a strictly positive operator. Then for all non-negative decreasing continuous function $h$ on $[0, \infty)$, we have
\begin{align*}
\|h(A^{-1})\| \leq h(\|A\|^{-1}).
\end{align*}
\end{lemma}
\begin{proof}
From $A\leq \|A\|I$, it follows that $\|A\|^{-1} I \leq A^{-1}$. That is $sp(A^{-1}) \subseteq (\|A\|^{-1}, \infty)$. So  $sp(h(A^{-1})) = h(sp(A^{-1}))\subseteq h(\|A\|^{-1}
, \infty)$. Since $h$ is decreasing, we have $h(A^{-1}) \leq h(\|A\|^{-1})I$ and therefore $\|h(A^{-1})\| \leq h(\|A\|^{-1})$.
\end{proof}
\begin{theorem} \label{th 3dec}
Let $A, B, X \in  \mathcal{B}(\mathcal{H})$ and $f, g$ are non-negative continuous functions on $[0, \infty)$ in which, $f(t)g(t)=t$, $(t\geq 0)$.\\
If $ 0<mI \leq B^{*}f^{2}(|X|)B$, $A^{*}g^{2}(|X^{*}|)A \leq MI$, $h:[0,\infty) \to[0, \infty)$ is an operator monotone decreasing function and $\sigma$ is an arbitrary mean between $\bigtriangledown$ and $!$, then for any unit vextor $x\in\mathcal{H}$,
\begin{align} \label{eq 3dec}
\left\|h\big(B^{*}f^{2}(|X|)B\big) \sigma h\big(A^{*}g^{2}(|X^{*}|)A\big)\right\| \leq \frac{m k}{M}  h\left(\left| \left\langle (A^{*}XB)x,x\right\rangle \right|\right),
\end{align}
where, $k=\dfrac{(M+m)^{2}}{4mM}$.\\
In particular,
\begin{align} \label{eq parti}
\left\|h\big(B^{*}f^{2}(|X|)B\big) \sigma h\big(A^{*}g^{2}(|X^{*}|)A\big)\right\| \leq  h\left(\left| \left\langle (A^{*}XB)x,x\right\rangle \right|\right).
\end{align}
\end{theorem}
\begin{proof}
Let $x\in\mathcal{H}$ be a unit vector. Now applying Lemma \ref{le k2}, AM-GM inequality and since every operator monotone decreasing function is operator convex \cite{Ando}, we have 
\begin{align*}
\frac{m}{M}h\left(\left| \left\langle A^{*}XBx,x\right\rangle \right|\right) \nonumber &=  \frac{m}{M} h\left(\left| \left\langle XBx,Ax\right\rangle \right|\right) \\ \nonumber & \geq  \frac{m}{M} h\left(\sqrt{\left\langle B^{*}f^{2}(|X|)Bx, x\right\rangle \left\langle A^{*}g^{2}(|X^{*}|)Ax, x\right\rangle}\right) \\ \nonumber & \geq  \frac{m}{M} h\left(\left\langle \left(\frac{B^{*}f^{2}(|X|)B + A^{*}g^{2}(|X^{*}|)A}{2}\right)x, x\right\rangle\right) \\ \nonumber &  \geq h\left( \frac{m}{M} \left\langle \left(\frac{B^{*}f^{2}(|X|)B + A^{*}g^{2}(|X^{*}|)A}{2}\right)x, x\right\rangle\right) \\  & \geq h\left( \frac{m}{M}\left\|\left(\frac{B^{*}f^{2}(|X|)B + A^{*}g^{2}(|X^{*}|)A}{2}\right)\right\|\right)    
\end{align*}
By hypothesis and operator convexity of $t \mapsto t^{-1}$, we obtain,
\begin{align*} 
\left\|\left(\frac{B^{*}f^{2}(|X|)B + A^{*}g^{2}(|X^{*}|)A}{2}\right)\right\| \leq M
\end{align*}
 and
\begin{align*}
  \left\|\left(\frac{B^{*}f^{2}(|X|)B + A^{*}g^{2}(|X^{*}|)A}{2}\right)^{-1}\right\| \leq \frac{1}{m}
\end{align*}
Therefore
\begin{align} \label{eq sat2}
\left\|\left(\frac{B^{*}f^{2}(|X|)B + A^{*}g^{2}(|X^{*}|)A}{2}\right)\right\| \leq \frac{M}{m} \left\|\left(\frac{B^{*}f^{2}(|X|)B + A^{*}g^{2}(|X^{*}|)A}{2}\right)^{-1}\right\|^{-1}
\end{align}
By using inequality \eqref{eq sat2} and Lemma \ref{le h2} , we have
\begin{align*}
& h\left(\frac{m}{M}\left\|\left(\frac{B^{*}f^{2}(|X|)B + A^{*}g^{2}(|X^{*}|)A}{2}\right)\right\|\right) \\  & \geq h\left(\left\|\left(\frac{B^{*}f^{2}(|X|)B + A^{*}g^{2}(|X^{*}|)A}{2}\right)^{-1}\right\|^{-1}\right) \\  & \geq \left\|h\left(\frac{B^{*}f^{2}(|X|)B + A^{*}g^{2}(|X^{*}|)A}{2}\right)\right\| \\ & \geq \frac{1}{k}  \left\|h\big(B^{*}f^{2}(|X|)B\big) \sigma h\big(A^{*}g^{2}(|X^{*}|)A\big)\right\|  
\end{align*}
where, in the last inequality, we used Lemma \ref{eq j} for $\nu = \frac{1}{2}$.
\end{proof}
\begin{remark}
In the assumptions of Theorem \ref{th 3dec}, we can replace $ 0<mI \leq B^{*}f^{2}(|X|)B$, $A^{*}g^{2}(|X^{*}|)A\leq MI$ with $0<mI \leq \frac{B^{*}f^{2}(|X|)B + A^{*}g^{2}(|X^{*}|)A}{2}\leq MI$.\\
 So, if we assume that $\frac{B^{*}f^{2}(|X|)B + A^{*}g^{2}(|X^{*}|)A}{2}$ is invertible, we can conclude \eqref{eq parti}.\\
  Similarly, if $\frac{B^{*}f^{2}(|X|)B + A^{*}g^{2}(|X^{*}|)A}{2}$ is not invertible, we can prove that 
 \[\left\|h\big(B^{*}f^{2}(|X|)B + \epsilon I\big) \sigma h\big(A^{*}g^{2}(|X^{*}|)A + \epsilon I\big)\right\| \leq  h\left(\left| \left\langle (A^{*}XB)x,x\right\rangle \right|\right)\]
 and taking limit of $\epsilon\to 0$, we can conclude \eqref{eq parti} without the assumption $ 0<mI \leq B^{*}f^{2}(|X|)B$, $A^{*}g^{2}(|X^{*}|)A \leq MI$.  
\end{remark}

\begin{remark}
Under the assumptions of Theorem \ref{th 3dec}, if $!_{\nu} \leq \sigma_{\nu} \leq \bigtriangledown_{\nu}$ and 
\begin{align*}
  0<mI \leq (B^{*}f^{2}(|X|)B)^{\frac{1}{1-\nu}}, (A^{*}g^{2}(|X^{*}|)A)^{\frac{1}{\nu}} \leq MI,
 \end{align*}
 then by applying \eqref{eq 002} for the concave function $t^{\nu}$ ($0<\nu<1$) and AM-GM inequality, respectively, we can write
\begin{align*}
\frac{m}{M}h\left(\left| \left\langle A^{*}XBx,x\right\rangle \right|^{2}\right)  & \geq  \frac{m}{M} h\left(\left\langle B^{*}f^{2}(|X|)Bx, x\right\rangle \left\langle A^{*}g^{2}(|X^{*}|)Ax, x\right\rangle\right) \\  & \geq  \frac{m}{M} h\left(\left\langle (B^{*}f^{2}(|X|)B)^{\frac{1}{1-\nu}}x, x\right\rangle^{1-\nu} \left\langle (A^{*}g^{2}(|X^{*}|)A)^{\frac{1}{\nu}} x, x\right\rangle^{\nu}\right) \\  & \geq  \frac{m}{M} h\left(\left\langle (1-\nu)(B^{*}f^{2}(|X|)B)^{\frac{1}{1-\nu}}
 + \nu (A^{*}g^{2}(|X^{*}|)A)^{\frac{1}{\nu}} x, x\right\rangle\right)
\end{align*}
Therefore, by similar argument to the proof of Theorem \ref{th 3dec}, we obtain
\begin{align} \label{eq 303dec}
\left\|h\big((B^{*}f^{2}(|X|)B)^{\frac{1}{1-\nu}}\big) \sigma_\nu h\big((A^{*}g^{2}(|X^{*}|)A)^{\frac{1}{\nu}}\big)\right\| \leq \frac{m k}{M}  h\left(\left| \left\langle (A^{*}XB)x,x\right\rangle \right|^{2}\right)
\end{align}
\end{remark}

\begin{lemma}\cite{Aujla} \label{eq auj}
If $A, B$ are positive operators and $f$ is a non-negative non-decreasing convex function on $[0,\infty)$, then 
\begin{align*}
\|f\big((1-\nu)A + \nu B\big)\| \leq \|(1-\nu)f(A) + \nu f(B) \|,
\end{align*}
for all $0<\nu<1$.
\end{lemma}
Applying Theorem \ref{th 3dec} to the decreasing convex function $h(t)=t^{-1}$ and $\sigma=\bigtriangledown$, we reach the following corollary:

\begin{corollary} \label{eq chand}
Let $A, B, X \in  \mathcal{B}(\mathcal{H})$ and $f, g$ are non-negative continuous functions on $[0, \infty)$ satisfying $f(t)g(t)=t$, $(t\geq 0)$. If\\
$0<mI\leq B^{*}f^{2}(|X|)B$, $A^{*}g^{2}(|X^{*}|)A\leq MI$, then
\begin{align} \label{eq omega}
\omega\big(A^{*}XB\big) \leq  \frac{m k}{M} \left\|B^{*}f^{2}(|X|)B  !  A^{*}g^{2}(|X^{*}|)A\right\|.
\end{align}
Furthermore, for increasing convex function $h^{'}:[0, \infty)\to[0, \infty)$, we have
\begin{align} \label{eq h^{'}}
h^{'}\left(\omega\big(A^{*}XB\big)\right) \leq \frac{m k}{2M}  \left\|h^{'}\left(B^{*}f^{2}(|X|)B\right) +  h^{'}\left(A^{*}g^{2}(|X^{*}|)A\right)\right\|.
\end{align}
In particular, for all $p \geq 1$
\begin{align} \label{eq p}
\omega^{p}\big(A^{*}XB\big) \leq \frac{m k}{2M}  \left\|\left(B^{*}f^{2}(|X|)B\right)^{p} + \left(A^{*}g^{2}(|X^{*}|)A\right)^{p}\right\|. 
\end{align}
\end{corollary}
\begin{proof}
Let us prove \eqref{eq h^{'}}. By inequalities \eqref{eq omega} and Lemma \ref{eq auj}, we get
\begin{align*}
h^{'}\left(\omega\big(A^{*}XB\big)\right) &\leq h^{'}\left(\frac{m k}{2M} \left\|B^{*}f^{2}(|X|)B  + A^{*}g^{2}(|X^{*}|)A\right\|\right) \\ & \leq \frac{m k}{M} h^{'}\left(\left\| \frac{B^{*}f^{2}(|X|)B  +  A^{*}g^{2}(|X^{*}|)A}{2}\right\|\right) \\ & \leq \frac{m k}{M} \left\|h^{'}\left(\frac{B^{*}f^{2}(|X|)B  +  A^{*}g^{2}(|X^{*}|)A}{2}\right)\right\| \\ & \leq  \frac{m k}{2M} \left\|h^{'}\big(B^{*}f^{2}(|X|)B\big)  + h^{'}\big(A^{*}g^{2}(|X^{*}|)A\big)\right\|   
\end{align*}
The third inequality in the above inequalities follows from \eqref{eq 002} (in fact, a similar argument to the proof of Lemma \ref{le h2}, leads to equality). The latest inequality obtains from Lemma \ref{eq auj}.\\
By taking $h^{'}(t)=t^{p} (p>1)$, we reach inequality \eqref{eq p}.  
\end{proof}

By taking $f(t)=g(t)=t^{\frac{1}{2}}$ in inequalities \eqref{eq omega} we get a refinement of inequality \eqref{Shebrawi} for $p=1$, and if we put $f(t)=g(t)=t^{\frac{1}{2}}$ in \eqref{eq p}, we present a refinement of inequality \eqref{Shebrawi}.\\

Applying inequality  \eqref{eq 303dec} to the decreasing convex function $h(t)=t^{-1}$, one can reach the similar results as corollary \ref{eq chand} (we omit the detail).\\

Next, we need the following two lemmas. The first lemma in part (a), which contains a very useful numerical radius inequality, can be found in \cite{Yamazaki}. Part (b) is well-known (see \cite{Bhatia}) and two lemma concerning spectral radius inequalities was given in \cite{Omar}.  

\begin{lemma}\label{le2.1}
Let $A$ be an operator in $\mathcal{B}(\mathcal{H})$. Then
\begin{enumerate}
\item[(a)]
$\omega(A)=\sup_{\theta\in \mathbb{R}}\|{\rm Re}(e^{i\theta}A)\|=\frac{1}{2}\sup_{\theta\in \mathbb{R}}\|A+e^{i\theta}A^{*}\|$.
\item[(b)]
$\omega(\begin{bmatrix}
A&0\\
0&B
\end{bmatrix})=\max(\omega(A),\omega(B))$.
\end{enumerate}
\end{lemma}

\begin{lemma} \label{th Abu} 
Let $A_{1}, A_{2}, B_{1},B_{2} \in \mathcal{B}(\mathcal{H})$. Then
\begin{align*}
r(A_{1}B_{1} + A_{2}B_{2}) & \leq \frac{1}{2} (\omega(B_{1}A_{1}) + \omega(B_{2}A_{2}))\\ & + \frac{1}{2} \sqrt{(\omega(B_{1}A_{1}) - \omega(B_{2}A_{2}))^{2} + 4\|B_{1}A_{2}\| \|B_{2}A_{1}\|}.
\end{align*}
\end{lemma}

In the next theorem, we give an inequality similar to \eqref{Sattari}.

\begin{theorem}
Let $A, B \in \mathcal{B}(\mathcal{H})$. Then for all non-negative non-decreasing convex function $h$ on $[0, \infty)$, we have
\begin{align} \label{eq mox}
h\big(\omega(A^{*}B)\big) \leq \frac{1}{2} h(\|A\| \|B\|) + \frac{1}{2} h\big(\omega(BA^{*})\big).
\end{align} 
\end{theorem}
\begin{proof}
Let $\theta \in \mathbb{R}$. Letting $A_{1}=e^{i\theta}A^{*}$, $B_{1}=B$, $A_{2}=B^{*}$ and $B_{2}=e^{-i\theta}A$ in Lemma \ref{th Abu} we can write
\begin{align*}
\|Re(e^{i\theta}(A^{*}B))\| &= r(Re(e^{i\theta}(A^{*}B))  \\ & \leq \frac{1}{4} (\omega(BA^{*}) + \omega(AB^{*})) \\ &+ \frac{1}{4} \sqrt {(\omega(BA^{*}) - \omega(AB^{*}))^{2} + 4\|AA^{*}\| \|BB^{*}\|} \\  & = \frac{1}{2} \omega(BA^{*}) + \frac{1}{2} \|A\| \|B\|
\end{align*}
Hence, by Lemma \ref{le2.1} (a) and convexity of $h$, we get \eqref{eq mox}.
\end{proof}
\begin{example}
Letting $A= \begin{bmatrix}
1&0\\
-1&2
\end{bmatrix}$ and $B= \begin{bmatrix}
1&5\\
-1&2
\end{bmatrix}$. Since\\ 
$ \frac{1}{4} \|AA^{*} + BB^{*}\| = 7.5432$ and $\frac{1}{2} \|A|\|B| = 6.1962$, we can say that inequality \eqref{eq mox}, in this example, is a refinement of \eqref{Sattari}. 
\end{example} 
\begin{corollary}
Let $A, B \in \mathcal{B}(\mathcal{H})$. Then for all $p\geq 1$ we have
\begin{align*}
\omega^{p}(A^{*}B) \leq \frac{1}{2} \|A\|^{p}\|B\|^{p} + \frac{1}{2} \omega^{p}(BA^{*}).
\end{align*}
\end{corollary}
\begin{corollary}
Let $A \in \mathcal{B}(\mathcal{H})$, $A = U|A|$ be the polar decomposition of $A$, and $f$, $g$ be two non-negative continuous functions on $[0, \infty)$ such that $f(t)g(t)=t$ $(t\geq0)$ and let $\tilde{A}_{f,g}=f(|A|)Ug(|A|)$ be generalize the Aluthge transform of $A$. Then for all $p\geq1$, 
\[\omega^{p}(A) \leq \frac{1}{2} \|f(|A|)\|^{p} \|g(|A|)\|^{p} + \frac{1}{2} \omega^{p}(\tilde{A}_{f,g}).
\] 
\end{corollary}

The following lemma will be useful in thr proof of the next results.

\begin{lemma} \cite{Kittaneh1}\label{le2.2}
If $A_{1},A_{2},B_{1},B_{2},X$ and $Y$ are operators in $\mathcal{B}(\mathcal{H})$. Then
\begin{equation}\label{eq2.1}
2\big\|A_{1}XA_{2}^{*}+B_{1}YB_{2}^{*}\big\| \leq \left\|\begin{bmatrix}
A_{1}^{*}A_{1}X+XA_{2}^{*}A_{2}&A_{1}^{*}B_{1}Y+XA_{2}^{*}B_{2}\\
B_{1}^{*}A_{1}X+YB_{2}^{*}A_{2}&B_{1}^{*}B_{1}Y+YB_{2}^{*}B_{2}
\end{bmatrix}\right\|
\end{equation}
\end{lemma}

\begin{theorem}\label{th2.7}
Let $A,B,X\in \mathcal{B}(\mathcal{H})$. Then
\begin{equation}\label{eq 333}
\omega(A^{*}XB)\leq \frac
{1}{4}\|AA^{*}X+XBB^{*}\|+\frac{1}{2}\omega(\begin{bmatrix}
XBA^{*}&0\\
0& BA^{*}X
\end{bmatrix})
\end{equation}
\end{theorem}
\begin{proof}
Applying the first  inequality in Lemma \ref{le2.1} (a) and by letting $A_{1} = B_{2} = e^{i\theta}A^{*}$, $A_{2} = B_{1} = B^{*}$ and $Y = X^{*}$ in inequality \eqref{eq2.1}, we have
\begin{align*}
\omega(A^{*}XB) = & \sup_{\theta\in \mathbb{R}}\big\|\rm{Re}(e^{i\theta}A^{*}XB)\big\| \\& = \frac{1}{2} \sup_{\theta\in \mathbb{R}}\big\|e^{i\theta}A^{*}XB+e^{-i\theta}B^{*}X^{*}A\big\|\\
&\leq \frac{1}{4} \sup_{\theta\in \mathbb{R}}\left\|\begin{bmatrix}
AA^{*}X+XBB^{*}&e^{-i\theta}AB^{*}X^{*}+e^{i\theta}XBA^{*}\\
e^{i\theta}BA^{*}X+e^{-i\theta}X^{*}AB^{*}&BB^{*}X^{*}+X^{*}AA^{*}
\end{bmatrix}\right\|\\
&\leq \frac{1}{4} \sup_{\theta\in \mathbb{R}}\left\|
\begin{bmatrix}
AA^{*}X+XBB^{*}&0\\
0&BB^{*}X^{*}+X^{*}AA^{*}
\end{bmatrix}
\right\|\\
\quad + &\frac{1}{4} \sup_{\theta\in \mathbb{R}}\left\|
\begin{bmatrix}
0&e^{i\theta}(XBA^{*}+e^{-2i\theta}AB^{*}X^{*}\\
e^{i\theta}(BA^{*}X+e^{-2i\theta}X^{*}AB^{*}&0
\end{bmatrix}
\right\|\\
&=\frac{1}{4}\big\|AA^{*}X+XBB^{*}\big\|\\
\qquad+ & \frac{1}{4} \sup_{\theta\in \mathbb{R}} \left(\max\big\{\big\|XBA^{*}+e^{-2i\theta}AB^{*}X^{*}\big\|,
\big\|BA^{*}X+e^{-2i\theta}X^{*}AB^{*}\big\|\big\}\right)
\end{align*}
Using the second equality in Lemma \ref{le2.1} (a), (b), respectively, we deduce the desired inequality \eqref{eq 333}. 
\end{proof}

\begin{remark}\label{re2.10}
By letting $X=I$ in the inequality \eqref{eq 333}, and by using Lemma \ref{le2.1} (b), it is easy to see that the inequality \eqref{eq 333} generalizes inequality \eqref{Sattari} for $p = 1$.
\end{remark}

\begin{example}
Taking $A= \begin{bmatrix}
1&2\\
3&0
\end{bmatrix}$, $B= \begin{bmatrix}
3&4\\
1&5
\end{bmatrix}$ and $X= \begin{bmatrix}
1&2\\
0&1
\end{bmatrix}$. By an easy computation, we find that
\begin{align*} 
\frac{1}{2} \| A^{*} |X^{*}| A + B^{*} |X| B \| \approx 59.5407,
\end{align*} 
\[
\frac{1}{4}\|AA^{*}X+XBB^{*}\|+\frac{1}{2}\omega(\begin{bmatrix}
XBA^{*}&0\\
0& BA^{*}X
\end{bmatrix}) \approx57.7024
\]
and $\omega (A^{*} X B) \approx 42.2677$. This show that the inequality \eqref{eq 333}, in this example,  provides an improvement of the inequality \eqref{Shebrawi} for $p = 1$. 
\end{example}

Using Theorem \ref{th2.7}, we get the following result which inequality \eqref{eq ba2} was obtained in \cite{Bakh}.

\begin{corollary}\label{co2.10}
Let $A \in \mathcal{B}(\mathcal{H})$, $A = U|A|$ be the polar decomposition of $A$, and  $f,g$ be two non-negative continuous functions on $[0, \infty)$ such that $f(x)g(x)=x$ $(x\geq0)$ and let $\tilde{A}_{f,g}=f(|A|)Ug(|A|)$ be generalize the Aluthge transform of $A$. . Then for all non-negative and increasing convex function $h$ on $[0,\infty)$, we have
\begin{equation}\label{eq ba2}
h\big(\omega(A)\big) \leq \frac{1}{4} \left\| h\big(f^{2}(|A|)\big) + h\big(g^{2}(|A|)\big)\right\| + \frac{1}{2}  h\big(\omega (\tilde{A}_{f,g})\big). 
\end{equation}
\end{corollary}
\begin{proof}
Since 
\begin{align*}
\omega(A) = \omega( Ug(|A|)f(|A|)) = \omega(Ug(|A|UU^{*}f(|A|)).
\end{align*}
If we take $A^{*}=Ug(|A|)$, $X=U$ and $B=U^{*}f(|A|)$ in \eqref{eq 333}, we get
\begin{align*}
\omega(A) \leq \frac{1}{4} \big\| \big(f^{2}(|A|) + g^{2}(|A|)\big) U \| + \frac{1}{2}  \omega (\tilde{A}_{f,g}). 
\end{align*} 
By the fact that $\|U\| = 1$ and convexity of $h$, we obtain \eqref{eq ba2}.
\end{proof}
\begin{theorem}\label{th2.11}
Let $A,B,X\in \mathcal{B}(\mathcal{H})$. Then
\begin{align*}
\omega(A^{*}XB+B^{*}XA) &\leq \left(\frac{1}{2} \big(\|A\|^{2} + \|B\|^{2}\big) + \|AB^{*}\|\right) \omega(X) \\ &\leq (\|A\|\|B\|+\|AB^{*}\|)\omega(X)
\end{align*}
\end{theorem}
\begin{proof}
By using the first equality in Lemma \ref{le2.1} (a) and the fact that $\rm {Re}(e^{i\theta}(A^{*}XB+B^{*}XA))=A^{*}{\rm Re}(e^{i\theta}X)B+B^{*}{\rm Re}(e^{i\theta}X)A$ and  putting $A_{1}=B_{2}=A^{*}$, $X=Y={\rm Re}(e^{i\theta}X)$ and $A_{2}=B_{1}=B^{*}$ in inequality \eqref{eq2.1}, we get
\begin{align*}
&\sup_{\theta\in \mathbb{R}} \big\|{\rm Re}(e^{i\theta}(A^{*}XB+B^{*}XA))\big\|\\
\quad \leq &\frac{1}{2} \sup_{\theta\in \mathbb{R}} \left \|
\begin{bmatrix}
AA^{*}{\rm Re}(e^{i\theta}X)+{\rm Re}(e^{i\theta}X)BB^{*}& AB^{*}{\rm Re}(e^{i\theta}X)+{\rm Re}(e^{i\theta}X)BA^{*}\\
BA^{*}{\rm Re}(e^{i\theta}X)+{\rm Re}(e^{i\theta}X)AB^{*}&BB^{*}{\rm Re}(e^{i\theta}X)+{\rm Re}(e^{i\theta}X)AA^{*}
\end{bmatrix}
\right\|\\
\quad \leq &\frac{1}{2} \sup_{\theta\in \mathbb{R}} \left\|
\begin{bmatrix}
AA^{*}{\rm Re}(e^{i\theta}X)+{\rm Re}(e^{i\theta}X)BB^{*}&0\\
0&BB^{*}{\rm Re}(e^{i\theta}X)+{\rm Re}(e^{i\theta}X)AA^{*}
\end{bmatrix}
\right\|\\
\qquad +&\frac{1}{2} \sup_{\theta\in \mathbb{R}} \left\|
\begin{bmatrix}
0&AB^{*}{\rm Re}(e^{i\theta}X)+{\rm Re}(e^{i\theta}X)BA^{*}\\
BA^{*}{\rm Re}(e^{i\theta}X)+{\rm Re}(e^{i\theta}X)AB^{*}&0
\end{bmatrix}
\right\|
\end{align*}
Using the first equality in Lemma \ref{le2.1} (a), we obtain
\begin{align*}
\omega(A^{*}XB+B^{*}XA)\leq \frac{1}{2}(\|A\|^{2}+\|B\|^{2})\omega(X)+\|AB^{*}\|\omega(X)
\end{align*}
The second inequality, follows by AM-GM inequality.
\end{proof}

The above theorem is a refinement of the following numerical radius inequality that was proved in \cite{Hirzallah},
\begin{align*}
\omega(A^{*}XB+B^{*}XA) \leq 2\|A\|\|B\|\omega(X).
\end{align*}

The following lemma is due to Kittaneh \cite{Kittaneh 0}

\begin{lemma}\label{kittan}
Let $A, B \in \mathcal{B}(\mathcal{H})$ such that $|A|B=B^{*}|A|$. If $f$ and $g$ are nonnegative continuous function on $[0, \infty)$ satisfying $f(t)g(t)=t$ $(t\geq0)$, then for any vectors $x, y \in \mathcal{H}$
\begin{align*}
  |\langle AB x,y\rangle| \leq r(B) \|f(|A|)x\|  \|g(|A^{*}|)y\|. 
\end{align*} 
\end{lemma}
\begin{theorem}  \label{th alpha}
Let $A, B, X \in \mathcal{B}(\mathcal{H})$ satisfying $|A^{*}|X=X^{*}|A^{*}|$ and  $f,g$ be two non-negative continuous functions on $[0, \infty)$ such that $f(t)g(t)=t$ $(t\geq0)$. If $h$ is a nonnegative increasing convex function on $[0,\infty)$, then
\begin{align*}
h\big(\omega^{2}(A^{*}XB)\big) \leq \big\| (1-\nu) h\big(r^{2}(X) (B^{*}f^{2}(|A^{*}|)B)^{\frac{1}{1-\nu}}\big) + \nu h\big( r^{2}(X) g^{\frac{2}{\nu}}(|A|)\big) \big\| 
\end{align*}
for all  $0<\nu<1$. Moreover, in special case for $r(X) \leq 1$, we have
\begin{align*}
h\big(\omega^{2}(A^{*}XB)\big) \leq r^{2}(X) \big\| (1-\nu) h\big((B^{*}f^{2}(|A^{*}|)B)^{\frac{1}{1-\nu}}\big) + \nu h\big( g^{\frac{2}{\nu}}(|A|)\big) \big\|
\end{align*}
\end{theorem}
\begin{proof}
Setting $y=x$ in Lemma \ref{kittan} and using \eqref{eq 002} for the concave function $t^{\nu}$, respectively, we get  
\begin{align*}
 |\langle A^{*}XB x,x\rangle|^{2} & \leq r^{2}(X) \|f(|A^{*}|) Bx\|^{2} \|g(|A|)x\|^{2} \\&  \leq r^{2}(X)  \langle B^{*}f^{2}(|A^{*}|)Bx, x \rangle \langle g^{2}(|A|)x, x \rangle \\ & = r^{2}(X)  \left\langle \left(\big(B^{*}f^{2}(|A|)B\big)^{\frac{1}{1-\nu}}\right)^{1-\nu} x, x\right\rangle  \left\langle \big(\left(g^{2}(|A|)\right)^{\frac{1}{\nu}}\big)^{\nu} x, x\right\rangle  \\ & \leq r^{2}(X) \left\langle \left(B^{*}f^{2}(|A^{*}|)B\right)^{\frac{1}{1-\nu}} x, x\right\rangle^{1-\nu}  \left\langle (g^{2}(|A|))^{\frac{1}{\nu}} x, x\right\rangle^{\nu}\\ & \leq r^{2}(X)  \left\langle (1-\nu)\left(B^{*}f^{2}(|A^{*}|)B\right)^{\frac{1}{1-\nu}} + \nu g^{\frac{2}{\nu}}(|A|) x, x \right\rangle.
\end{align*}
Hence by taking the supremum over $x\in \mathcal{H}$, we get
\begin{align*}
\omega^{2}(A^{*} X B) \leq r^{2}(X)  \big\|(1-\nu) \left(B^{*}f^{2}(|A^{*}|)B\right)^{\frac{1}{1-\nu}} + \nu g^{\frac{2}{\nu}}(|A|)\big\|.
\end{align*}
Since $h$ is an increasing convex function, we have
\begin{align*}
h\big(\omega^{2}(A^{*} X B)\big) & \leq h\left(r^{2}(X)  \big\|(1-\nu) \left(B^{*}f^{2}(|A^{*}|)B\right)^{\frac{1}{1-\nu}} + \nu g^{\frac{2}{\nu}}(|A|)\big\|\right) \\ & =\big\| h\left(r^{2}(X)  (1-\nu) \left(B^{*}f^{2}(|A^{*}| )B\right) ^{\frac{1}{1-\nu}}+ \nu g^{\frac{2}{\nu}}(|A|)\right)\big\| \\ & \leq \big\| (1-\nu) h\big(r^{2}(X) (B^{*}f^{2}(|A^{*}|)B)^{\frac{1}{1-\nu}}\big) + \nu h\big( r^{2}(X) g^{\frac{2}{\nu}}(|A|)\big) \big\| 
\end{align*}
where, in the last inequality we used Lemma \ref{eq auj}. 
\end{proof}

Now we present some applications of Theorem \ref{th alpha}.\\ 

Letting $f(t)=t^{1-\nu}$ and $g(t)=t^{\nu}$ for $0< \nu<1$ in Theorem \ref{th alpha} we get

\begin{corollary} \label{eq 0g}
Let $A, B, X \in \mathcal{B}(\mathcal{H})$ satisfying $|A^{*}|X=X^{*}|A^{*}|$. If $h$ is a nonnegative increasing convex function on $[0,\infty)$, then for all  $0<\nu<1$
\begin{align*}
h\big(\omega^{2}(A^{*}XB)\big) \leq \big\|(1-\nu ) h\big(r^{2}(X) (B^{*}|A^{*}|^{2}B)\big) + \nu h\big(r^{2}(X) |A|^{2}\big)\big\|.
\end{align*}
Inparticullar, for $r(X) \leq1$
\begin{align*}
h\big(\omega^{2}(A^{*}XB)\big) \leq  r^{2}(X)  \big\|(1-\nu ) h (B^{*}|A^{*}|^{2}B\big) + \nu h (|A|^{2})\big\|.
\end{align*}
\end{corollary}

By the convexity $h(t)=t^{p}$ for $p\geq1$ we have

\begin{corollary} \label{eq 00g}
Let $A,B,X\in \mathcal{B}(\mathcal{H})$, then for all  $0<\nu<1$ and $p\geq1$
\begin{align*}
\omega^{2p}(A^{*}XB) \leq r^{2p}(X) \big\|(1-\nu) (B^{*}|A^{*}|^{2}B)^{p} + \nu  |A|^{2p}\big\|.
\end{align*}
\end{corollary}

In addition, by using Theorem \ref{th alpha} and corollaries \ref{eq 0g} , \ref{eq 00g}  for $X = B = I$, we obtain several generalization of inequality \ref{eq kit2}.

\end{document}